\documentclass[12pt,reqno,twoside]{amsart}
\usepackage{amsthm,amsmath,amssymb,mathrsfs,amsbsy,mathabx}
\usepackage{graphicx,epsfig,wrapfig,sidecap,subfig}
\usepackage[all]{xy}
\usepackage{proof}
\usepackage[usenames]{color}
\usepackage{bm}
\usepackage{listings}
\usepackage{hyperref}

\xyoption{dvips}
\xyoption{color}

\usepackage{color}

\title[BD non-equivalence in substitution tilings]{Bounded Displacement Non-Equivalence In Substitution Tilings}
\date{}
\author{Dirk Frettl\"oh}
\address{Universit\"at Bielefeld, Germany, {\tt dfrettloeh@techfak.uni-bielefeld.de}}
\author{Yotam Smilansky}
\address{The Hebrew University of Jerusalem, Israel$^*$, {\tt yotam.smilansky@rutgers.edu}}
\author{Yaar Solomon}
\address{Ben-Gurion University of the Negev, Israel, {\tt yaars@bgu.ac.il}\bigskip \newline \textit{$^*$Current affiliation}: Rutgers University, NJ, USA.}

\newcommand{\N}{{\mathbb{N}}}
\newcommand{\Z}{{\mathbb{Z}}}

\newcommand{\R}{{\mathbb{R}}}

\newcommand{\XX}{\mathbb{X}}
\newcommand{\YY}{\mathbb{Y}}

\renewcommand{\AA}{\mathcal{A}}
\newcommand{\FF}{{\mathcal{F}}}

\newcommand{\PP}{\mathcal{P}}
\newcommand{\QQ}{\mathcal{Q}}
\newcommand{\RR}{\mathcal{R}}
\newcommand{\TT}{\mathcal{T}}

\newcommand{\bd}{\stackrel{\scriptscriptstyle \rm BD}{\sim}}
\newcommand{\dotcup}{\mathaccent\cdot\cup}

\newcommand{\df}{{\, \stackrel{\mathrm{def}}{=}\, }}

\DeclareMathOperator{\supp}{\mathrm{supp}}

\newcommand{\absolute}[1] {\left|{#1}\right|}

\newcommand{\norm}[1]{\left\|{#1}\right\|}
\newcommand{\inpro}[2]{\langle{#1},{#2}\rangle}

\newcommand {\ignore}[1]  {}

\newcommand {\act}[1]{{#1}}

\theoremstyle{plain}
\newtheorem{thm}{Theorem}[section]

\newtheorem*{thmnonum}{Theorem}
\newtheorem{lem}[thm]{Lemma}
\newtheorem{prop}[thm]{Proposition}
\newtheorem{cor}[thm]{Corollary}

\theoremstyle{definition}
\newtheorem{definition}[thm]{Definition}

\numberwithin{equation}{section}
\swapnumbers

\newif\ifdraft\drafttrue
\usepackage{color}

\begin{document}
\begin{abstract}
      In the study of aperiodic order and mathematical models of quasicrystals, questions regarding equivalence relations on Delone sets naturally arise. This work is dedicated to the bounded displacement (BD) equivalence relation, and especially to results concerning instances of non-equivalence. We present a general condition for two Delone sets to be BD non-equivalent, and apply our result to Delone sets associated with tilings of Euclidean space. First we consider substitution tilings, and exhibit a substitution matrix associated with two distinct substitution rules. The first rule generates only periodic tilings, while the second generates tilings for which any associated Delone set is non-equivalent to any lattice in space. As an extension of this result, we introduce arbitrarily many distinct substitution rules associated with a single matrix, with the property that Delone sets generated by distinct rules are non-equivalent. We then turn to the study of mixed substitution tilings, and present a mixed substitution system that generates representatives of continuously many distinct BD equivalence classes.
\end{abstract}

\maketitle

\section{Introduction}\label{sec:introduction}

A point set $\Lambda \subset \R^d$ is called a \emph{Delone set} if it is both \emph{uniformly discrete} and \emph{relatively dense}, that is, if there exist constants $r$ and $R$ so that the intersection of $\Lambda$ with every ball of radius $r$  contains at most one point, and the intersection of $\Lambda$ with every ball of radius $R$ is non-empty. We say that two Delone sets $\Lambda$ and $\Lambda'$ are
\emph{bounded displacement (BD) equivalent}, and denote $\Lambda \bd \Lambda'$, if there exists a
\emph{scaling constant} $\alpha>0$ and a bijection $\phi: \Lambda \to \alpha \Lambda'$ satisfying
\begin{equation}\label{eq:BD_condition}
\sup\{\norm{x-\phi(x)} \mid x\in\Lambda\}<\infty,
\end{equation}
where $\norm{x}$ is the Euclidean norm of a vector $x\in\R^d$.  A bijection $\phi$ satisfying \eqref{eq:BD_condition} is called a \emph{BD-map}. Note that unlike the definition of BD equivalence, in the definition of a BD-map no scaling is allowed. For general discrete sets, such maps were previously studied by Duneau and Oguey \cite{DO1,DO2}, Laczkovich \cite{Laczk} and Deuber, Simonovits and S\'{o}s \cite{DSS95}.

BD equivalence is often defined under the restriction $\alpha=1$, which is perhaps a more intuitive definition, since under this restriction the equivalence of $\Lambda$ and $\Lambda'$ implies that $\Lambda$ can be transformed into $\Lambda'$ by translating 
each point individually by some uniformly bounded amount. Nevertheless, our less restrictive definition ``mods out'' scaling issues, and so is better suited for the study of structural properties of Delone sets. For instance, two lattices of different covolume in $\R^d$ cannot admit a BD-map, and therefore are not equivalent under the restrictive definition, whereas under our definition of the BD equivalence relation all lattices in $\R^d$ belong to the same class, see \cite{DO2} or \cite{HKW}.

In this work we focus on instances of BD non-equivalence. In the context of Delone sets, the study of the BD equivalence relation is frequently paired with that of the less restrictive bi-Lipschitz equivalence, where two Delone sets $\Lambda$ and $\Lambda'$ in $\R^d$ are called \emph{bi-Lipschitz (BL) equivalent} if they admit a bi-Lipschitz bijection between them. Since BD equivalence of Delone sets implies BL equivalence, the study of BD non-equivalence is also motivated by questions on BL non-equivalence, which are often highly non-trivial. The mere existence of Delone sets which are not BL equivalent to a lattice was regarded as an interesting open problem until it was finally established independently by Burago and Kleiner \cite{BK1}  and by McMullen \cite{McMullen}. Explicit constructions were later presented in \cite{CortezNavas} and in \cite{Garber}, and it was shown in \cite{Magazinov} that in fact there exist continuously many BL equivalence classes. Nevertheless, there are still many open problems concerning BL non-equivalence.
\\

Recall that a Delone set $\Lambda$ is \textit{periodic} if the group of its translational symmetries $\{ t \in \R^d \mid \Lambda + t = \Lambda \}$ is cocompact, and \emph{non-periodic} if this group is trivial. It was shown in \cite{DO1} that all periodic Delone sets are BD equivalent to each other. Hence, in order to produce examples that are BD non-equivalent to any lattice, it is promising to study non-periodic Delone sets. 

An especially interesting and well studied family of non-periodic Delone sets consists of point sets associated with substitution tilings, defined by placing a point in every tile. A substitution tiling is defined by a substitution rule on a finite set of prototiles, to which the associated Delone set is said to correspond. A recent extension of the construction of substitution tilings is that of mixed substitution tilings, which involves the mixing of several distinct substitution rules to define tilings of the space. The study of related constructions includes \cite{GaeMal}, \cite{Rust} and \cite{SchmiedingTrevino}. All required definitions for substitution and mixed substitution rules and tilings, which are the main objects of interest in this work, are given in \S \ref{sec:preliminaries}. For a comprehensive introduction to substitution tilings, as well as to the field of aperiodic order in general, we refer to \cite{BaakeGrimm}.  
 
Early results concerning the BD equivalence relation and Delone sets associated with substitution tilings include \cite{ACG}, \cite{HoZ} and \cite{Solomon11}. These results were further improved by the third author in \cite{Solomon14} where the following criterion was established. 

\begin{thmnonum}\cite[Theorem 1.2]{Solomon14}
	Let $\varrho$ be a primitive substitution rule on polygonal prototiles in $\R^d$, and let $\Lambda$ be a corresponding Delone set. Denote by $M_{\varrho}$ the associated substitution matrix and its eigenvalues by $\lambda_1>\absolute{\lambda_2}\ge \ldots \ge \absolute{\lambda_n}$, see \S\ref{sec:preliminaries} for precise definitions. Let $t\ge 2$ be the minimal index for which the eigenvalue $\lambda_t$ has an eigenvector with non-zero sum of coordinates. 
	\begin{itemize}		
		\item[(I)]
		If $\absolute{\lambda_t}<\lambda_1^{\frac{d-1}{d}}$ then $\Lambda$ is BD equivalent to a lattice. 
		
		\item[(II)]
		If $\absolute{\lambda_t}>\lambda_1^{\frac{d-1}{d}}$ then $\Lambda$ is not BD equivalent to any lattice.
	\end{itemize} 
\end{thmnonum}
 
This and previous results all rely on a criterion by Laczkovich \cite[Theorem 1.1]{Laczk}, which gives a necessary and sufficient condition for BD equivalence to $\Z^d$.

\subsection{Statement of Results}

Theorem \ref{thm:non_BD_criterion} below establishes a sufficient condition for BD non-equivalence. Note the resemblance to the sufficient condition given by \cite[Lemma 2.3]{Laczk} for two discrete sets to be BD equivalent. The proof of our Theorem \ref{thm:non_BD_criterion} is along the lines of the proof of \cite[Theorem 1.1]{Laczk}, where Laczkovich only gives the result for $\Lambda_1$ being a lattice. 

Given $x=(x_1,\ldots,x_d) \in \R^d$, denote by $C(x)$ the axis-parallel unit cube $\bigtimes_{i=1}^d [x_i-\frac{1}{2},x_i+\frac{1}{2})$ centered at $x$. 
Denote by $\QQ_d$ the set $\left\{C(x) \mid x\in \Z^d \right\}$ of lattice centered unit cubes, and let $\QQ_d^*$ be the collection of all finite unions of elements of $\QQ_d$. Here and in what follows $\mu_s$ denotes the $s$-dimensional Lebesgue measure in $\R^d$, and $\#S$ the cardinality of a set $S$.

\begin{thm}\label{thm:non_BD_criterion}
	Let $\Lambda_1, \Lambda_2$ be two Delone sets in $\R^d$ 
	and suppose there is a sequence $(A_m)_{m\in\N}$ of sets $A_m\in \QQ_d^*$ for which
	\begin{equation}\label{eq:non_BD_condition}
	\lim_{m\to\infty} \frac{| \#(\Lambda_1 \cap A_m) - \#(\Lambda_2 \cap A_m) |}{\mu_{d-1}(\partial A_m)} = \infty.
	\end{equation}   
	Then there is no BD-map $\phi:\Lambda_1 \to \Lambda_2$.
\end{thm}
 

The conditions in  \cite[Theorem 1.2]{Solomon14} for Delone sets associated with substitution tilings depend on information that can be deduced from the substitution matrix alone. We are lead to ask, are Delone sets in $\R^d$ that correspond to substitution rules with the same substitution matrix necessarily BD equivalent? 

\begin{thm}\label{thm:period+nonperiodic_same_matrix}
There exists a set of prototiles $\mathcal{F}$ in $\R^2$ and primitive substitution rules $\varrho_1, \varrho_2$ on $\mathcal{F}$ with $M_{\varrho_1}=M_{\varrho_2}$, so that 
\begin{itemize}
\item[(i)]
Tilings defined by $\varrho_1$ are periodic. In particular, any associated Delone set is BD equivalent to a lattice. 
\item[(ii)]
Any Delone set corresponding to $\varrho_2$ is BD non-equivalent to any lattice. 
\end{itemize}
In particular, if $\Lambda_1$ corresponds to $\varrho_1$ and $\Lambda_2$ corresponds to $\varrho_2$, then $\Lambda_1\not\bd\Lambda_2$. 
\end{thm}
Observe that in view of  \cite[Theorem 1.2]{Solomon14} stated above, the equality $\absolute{\lambda_t}=\lambda_1^{(d-1)/d}$ must hold for the substitution rules $\varrho_1$ and $\varrho_2$.\\

The next result concerns mixed substitution tilings, see \S \ref{sec:preliminaries} for precise definitions. Consider a mixed substitution system $\Sigma=(\sigma_1,\ldots,\sigma_k)$ of $k$ primitive substitution rules on a single set of prototiles $\FF$ in $\R^d$. Let $w\in\Omega=\{1,\ldots,k\}^\N$ be an infinite word. Define the family of patches of tiles 
\begin{equation}
\mathcal{P}_w \df \left\{\sigma_{w_1} (\sigma_{w_2} ( \cdots \sigma_{w_m}(T) \cdots)) \mid T\in\FF, m\in\N \right\}.\notag
\end{equation}
Denote  by $\XX_{\Sigma,w}$ the space of tilings $\TT$ with the property that every patch $P \subset \TT$ is a translate of a subset of an element of $\mathcal{P}_w$, and let $\XX_{\Sigma,\Omega}=\bigcup_{w\in\Omega} \XX_{\Sigma,w}$.

\begin{thm}\label{thm:continuously_many_things}
There exists a set of prototiles $\FF$ in $\R^2$ and a mixed substitution system $\Sigma=(\sigma_1,\sigma_2)$ with $M_{\sigma_1}=M_{\sigma_2}$, so that  $\XX_{\Sigma,\Omega}$ contains continuously many tilings with pairwise BD non-equivalent associated Delone sets. Moreover every Delone set associated with a tiling in $\XX_{\Sigma,\Omega}$ is BL equivalent to a lattice. In particular, all of these distinct BD class representatives belong to the same BL class. 
\end{thm} 

\noindent{\bf Remark.} A similar result holds for right-mixed substitution tilings, see \S\ref{sec:right-left}. In addition, consider the space $\XX_{\text{rand}(\Sigma)}$ of tilings $\TT$ with the property that every finite patch $P \subset \TT$ is a translate of a subset of a patch which occurs in some $\Sigma$-iterate of some $T \in \FF$, where now the
$\sigma_i$'s are applied at random individually on each tile (see
\cite[Definitions 9 \& 12]{GRS} for a formal definition). Clearly $\XX_{\Sigma,w} \subset \XX_{\text{rand}(\Sigma)}$
for every $w\in\Omega$, and so the result holds also for $\XX_{\text{rand}(\sigma_1, \sigma_2)}$.\\

The arguments developed for the proof of Theorem \ref{thm:continuously_many_things} allow us to deduce the following result concerning standard non-mixed substitution tilings. 
\begin{thm}\label{thm:arbitrary-many-things}
	There exists a set of prototiles $\mathcal{F}$ in $\R^2$ so that for every $q \in \N$ there exist $q$ primitive substitution rules $\varrho_1, \ldots, \varrho_q$ on $\mathcal{F}$, all with the same substitution matrix $M_{\varrho}$, and $q$ tilings $\TT_1,\ldots,\TT_q$ defined by these rules, so that the associated Delone sets are pairwise BD non-equivalent.
\end{thm}

A final curious consequence of our arguments is the demonstration of a Delone set $\Lambda$ associated with a substitution tiling of $\R^2$, with the property that $\Lambda$ is BD non-equivalent to the set $-\Lambda=\{-x \mid x \in \Lambda \}$, see Corollary \ref{cor:rotation_by_pi_not_BD}.

\subsection*{Acknowledgments} 
We are happy to thank Dan Rust and Barak Weiss for helpful discussions, and the two anonymous reviewers for their thoughtful comments. We thank The Center For Advanced Studies In Mathematics in Ben-Gurion University and the Research Centre of Mathematical Modelling
(RCM$^2$) at Bielefeld University for supporting the visit of the first author in Israel. The second author is grateful for the support of the David and Rosa Orzen Endowment Fund via an Orzen Fellowship, and the ISF grant No. 1570/17.


\section{Preliminaries and notations}\label{sec:preliminaries}
	 
In the general setting of tiling theory, a \emph{tile} $T\subset\R^d$ refers to a set with some regularity assumptions on it, such as being equal to the closure of its interior or being homeomorphic to a closed $d$-dimensional ball. In the constructions described in this work, all tiles are either squares or rectangles in the plane, and so all standard regularity assumptions hold.   

A \emph{tessellation} of a set $U\subset\R^d$ is a collection of tiles with pairwise disjoint interiors such that their union is equal to $U$. A tessellation $P$ of a bounded set $U\subset\R^d$ is called a \emph{patch}, and the set $U$ is called the \emph{support of $P$} and denoted by $\supp(P)$. A tessellation of $\R^d$  is called a \emph{tiling} and denoted by $\TT$.  A Delone set $\Lambda$ is said to be  \emph{associated} with a tiling $\TT$ if $\Lambda$ can be obtained from $\TT$ by placing a point in each tile, so that every tile is represented by a distinct point in $\Lambda$.

Two tiles are \emph{translation equivalent} if they differ by a translation, and we denote by $\FF$ the set of representatives of translation equivalence classes. We assume throughout that $\mathcal{F}=\{T_1,\ldots,T_n\}$ is a finite set. Elements of $\FF$ are called \emph{prototiles}, and if a tile $T$ is equivalent to a prototile $T_i\in\FF$ it is also called a tile \emph{of type} $i$. The translation equivalence relation is extended to patches of tiles from $\FF$, and the corresponding set of representatives is denoted by  $\FF^*$.


\subsection{Substitution Tilings}\label{subsec:substitution_tilings}
Let $\{T_1,\ldots,T_n\}$ be a set of tiles in $\R^d$. We construct patches and tilings with prototiles $\FF=\{T_1,\ldots,T_n\}$. 

\begin{definition}\label{def:SubRule}
	A \emph{substitution rule} on $\FF$ with \emph{inflation factor}  $\xi>1$ is a map $\varrho:\mathcal{F}\to\mathcal{F}^*$ satisfying $\xi T_i=\supp(\varrho(T_i))$ for every $i$. Namely, it is a fixed set of tessellations of the sets $\{\xi T_1,\ldots,\xi T_n\}$ by tiles from $\FF$. 
	The domain of the map $\varrho$ is extended to $\FF^*$ by applying $\varrho$ individually to each tile, and similarly to tilings for which all patches belong to $\FF^*$. 
\end{definition}

\begin{definition}\label{def:SubsTiling}
	Let $\varrho$ be a substitution rule on $\mathcal{F}$, and consider the patches 
	\[\PP_{\varrho}\df\left\{\varrho^m(T) \mid m\in\N \: , \: T\in\mathcal{F} \right\}. \] 
	The \emph{substitution tiling space} $\XX_{\varrho}$ is the collection of all tilings $\TT$ of $\R^d$ with the property that for every patch $P\subset\TT$, there is a patch $P'\in\PP_{\varrho}$ that contains a translated copy of $P$ as a sub-patch. 
	Elements of $\XX_{\varrho}$ are called \emph{substitution tilings} defined by $\varrho$. 
\end{definition}
\begin{definition}\label{def:SubMatrix+Primitive}
	Let $\varrho$ be a substitution rule on $\mathcal{F}$. The entries $a_{ij}$ of the  \emph{substitution matrix} $M_{\varrho}\in M_n(\Z)$ associated with $\varrho$ are given by 
	\[a_{ij} = \#\left\{\text{Tiles of type } i \text{ in } \varrho(T_j) \right\}.\] 
\end{definition}   

We say that $\varrho$ is \emph{primitive} if $M_{\varrho}$ is a primitive matrix, that is, if there exists a power of $M_{\varrho}$ with strictly positive entries. Clearly $\xi^d$ is an eigenvalue of $M_\varrho$, and since the vector $u_1\in\R^n$ whose $i$'th entry is the volume of $T_i$ is a corresponding positive left eigenvector, primitivity and the Perron-Frobenius Theorem implies that the eigenvalues of $M_\varrho$ can be ordered so that $\lambda_1 = \xi^d>\absolute{\lambda_2} \ge\ldots\ge\absolute{\lambda_n}$. 

\begin{definition}
The \emph{natural density} of a Delone set $\Lambda\subset\R^d$ is defined to be
\begin{equation}\label{eq:natural_density}
\lim_{r\to\infty}\frac{\#(\Lambda \cap B_r(0))}{\mu_d\left(B_r(0)\right)},\notag
\end{equation}
provided the limit exists, where $B_r(x)$ is the open ball of radius $r$ about $x\in\R^d$.
\end{definition}

Given a primitive substitution rule $\varrho$, the Perron-Frobenius Theorem also implies that $M_\varrho$ has a positive right eigenvector $v_1\in\R^n$ associated with its leading eigenvalue $\xi^d$, and a direct computation shows that the natural density of any Delone set corresponding to $\varrho$ exists and is equal to  
\begin{equation}\label{eq:alpha_def}
\alpha = \frac{\inpro{\mathbf{1}}{v_1}}{\inpro{u_1}{v_1}},
\end{equation}
where $\inpro{\cdot}{\cdot}$ is the standard inner product in $\R^n$, see also \cite{BaakeGrimm} or \cite[\S 2]{Solomon14}.  

\subsection{Mixed Substitution Tilings}

Fix a set $\FF$ of $n$  prototiles in $\R^d$ and a \emph{mixed substitution system} $\Sigma = (\sigma_1,\ldots,\sigma_k)$ of $k$ primitive substitution rules on $\FF$. Denote $\AA = \{1,\ldots,k\}$. Then $\Omega=\AA^\N$ is the space of \emph{infinite words} and $\AA^*=\{ \AA^m \mid
m \in \N\}$ is the set of \emph{finite words} over $\AA$. Given $w=(w_1,w_2\ldots)\in\Omega$, the word $(w_1,\ldots,w_m)\in\AA^m$, denoted by $w(m)$, is the length $m$ prefix of $w$.

The semi-group $(\AA^*, \text{concatenation})$ acts on $\FF^*$ with respect to the mixed substitution system $\Sigma$ by
\begin{equation}\label{eq:action_of_a_finite_word}
\act{a}.P \df \sigma_{a_1} (\sigma_{a_2} ( \cdots \sigma_{a_m}(P) \cdots)),
\end{equation}
where $a=(a_1,\ldots,a_m)\in \AA^m$ is a finite word and $P$ is a patch in $\FF^*$. This can be viewed as an action from the left, see \S\ref{sec:right-left} for further discussions.

A  system $\Sigma$ is \textit{uniformly primitive} if there exists $m_0\in\N$ so that for every $a\in\AA^{m_0}$ and every $T\in\FF$ the patch $\act{a}.T$ contains tiles of all types. Clearly in such a case the same holds for any $m\ge m_0$. 

\begin{definition}
Fix a word $w \in \Omega$ and consider the family of patches 
\begin{equation}\label{eq:PP_w}
\mathcal{P}_w \df \left\{ \act{w(m)}.T \mid T\in\FF, m\in\N \right\}.\notag
\end{equation}
The \emph{mixed substitution tiling space} $\XX_{\Sigma,w}$ is the collection of tilings $\TT$ of $\R^d$ with the property that for every patch $P\subset\TT$ there is a patch $P'\in\PP_w$ that contains a translated copy of $P$ as a sub-patch. 
Elements of $\XX_{\Sigma,w}$ are called \emph{mixed substitution tilings} defined by $w$ and $\Sigma$. The space of all tilings defined by $\Sigma$ is denoted by $\XX_{\Sigma,\Omega}=\bigcup_{w\in\Omega} \XX_{\Sigma,w}$.

\end{definition}


\section{BD equivalent Delone sets and Hall's marriage theorem}\label{sec:Hall}

Recall that $\QQ_d^*$ is the collection of unions of finite subsets of $\QQ_d$, the set of all axis-parallel unit cubes in $\R^d$ centered at points of $\Z^d$. For a set $A\subset \R^d$ and $s>0$ we denote 
\[A^{+s} = \left\{x\in \R^d \mid \exists a\in A \text{ such that } \norm{x-a}\le s \right\}.\]

\begin{prop}\label{prop:BD_criterion}
	Let $\Lambda_1, \Lambda_2$ be Delone sets in $\R^d$. Then there  exists a BD-map $\phi:\Lambda_1 \to \Lambda_2$ if and only if there exists $s>0$ such that for every $U\in\QQ_d^*$   
	\begin{align}\label{eq:Hall's_condition_with cubes}
	 \#(\Lambda_1 \cap U) \le \#(\Lambda_2 \cap U^{+s}) \quad
	 \mbox{ and }  \quad
	 \#(\Lambda_2 \cap U) \le \#(\Lambda_1 \cap  U^{+s})
	\end{align}
\end{prop}
 
\begin{proof}
	Fix $m>0$, and consider the bipartite graph
	\[ G_m = \big(V_1\dotcup V_2,E_m\big) = \big( \Lambda_1 \dotcup \Lambda_2, \big\{ \{x,y\}
	\mid x \in \Lambda_1, y \in \Lambda_2, \norm{x-y} \le m \big\} \big).\]
	Given a finite subset $F \subset \Lambda_1$ denote its neighborhood in $G_m$ by
	\[N_{G_m}(F)\df\left\{ y\in\Lambda_2 \mid 	\exists x \in F: \; \{x,y\} \in E_m \right\}, \]
	and similarly for finite subsets of $\Lambda_2$. Observe that the existence of a perfect matching in $G_m$ for some $m$ is equivalent to the existence of a BD-map from $\Lambda_1$ to $\Lambda_2$. By an infinite version of Hall's Marriage Theorem due to Rado \cite{Rado}, the graph $G_m$ contains a perfect matching if and only if for all finite subsets $F_1 \subset \Lambda_1$ and $F_2 \subset \Lambda_2$ the inequalities
	\begin{equation}\label{eq:Hall's_condition}
		\#F_1 \le \#N_{G_m}(F_1)  \text{ and } \#F_2 \le \#N_{G_m}(F_2)
	\end{equation}
	hold. It is thus left to show that there exists $s>0$ for which \eqref{eq:Hall's_condition_with cubes} holds for any $U\in\QQ_d^*$ if and only if there exists $m>0$ for which \eqref{eq:Hall's_condition} holds for all finite
	subsets $F_1 \subset \Lambda_1$ and $F_2 \subset \Lambda_2$.
	
	First, assume $s>0$ is such that \eqref{eq:Hall's_condition_with cubes} holds for every $U\in\QQ_d^*$. Given finite sets $F_1 \subset \Lambda_1$ and $F_2 \subset \Lambda_2$ define the sets $U_1, U_2\in\QQ_d^*$ by 
	\begin{equation}\label{eq:U_1,U_2}
	U_i = \bigcup \left\{C(x) \mid F_i\cap C(x)\neq\varnothing, x\in\Z^d \right\}, \quad i=1,2.
	\end{equation}
	The diameter of a cube in $\QQ_d$ is $\sqrt{d}$, and so $U_i\subset F_i^{+\sqrt{d}}$. Then for $m = s+\sqrt{d}$ 
	\begin{equation}\label{eq:s_to_m}
	U_i^{+s} 
	\subset 
	\{x\in\R^d \mid  \exists a\in F_i^{+\sqrt{d}} \text{ so that } \norm{a-x}\le s \} \subset F_i^{+m}.
	\end{equation}
	By \eqref{eq:Hall's_condition_with cubes}, and in view of \eqref{eq:U_1,U_2} and \eqref{eq:s_to_m}, we have 
	\begin{equation}
	\#F_1  =
	 \#(\Lambda_1\cap U_1) \le 
	\#(\Lambda_2\cap U_1^{+s}) \le 
	 \#(\Lambda_2\cap F_1^{+m})
	\le  \#N_{G_m}(F_1),\notag
	\end{equation}
	and similarly $\#F_2 \le \#N_{G_m}(F_2)$, as required.
	
	Conversely, let $m>0$ be so that \eqref{eq:Hall's_condition} holds for all finite subsets $F_i \subset \Lambda_i$. Given $U\in\QQ_d^*$, set $F_i=\Lambda_i \cap U$. Clearly $\#N_{G_m}(F_i)\subset U^{+m}$, and so by \eqref{eq:Hall's_condition}, condition \eqref{eq:Hall's_condition_with cubes} holds with $s=m$.  
\end{proof}

Proposition \ref{prop:BD_criterion} implies the following simple but useful result.

\begin{cor}\label{cor:different_densities_implies_not_BD}
	Let $\Lambda_1$ and $\Lambda_2$ be Delone sets in $\R^d$ with natural density $\alpha_1$ and $\alpha_2$ respectively, where $\alpha_1\neq\alpha_2$. Then there is no BD-map $\phi:\Lambda_1 \to \Lambda_2$. 	
\end{cor}

\begin{lem}\label{lem:boundary_est.}
	Let $\Lambda\subset\R^d$ be a Delone set and let $k\in\N$ be so that every box in $\frac1k \QQ_d$ contains at most one point of $\Lambda$. Given $s>0$, set $c(s) = 2\cdot 3^{d-1}(1+2(s\sqrt{d}+d))^d$. Then for every $s>0$ and $U \in \QQ_d^*$ we have 
	\begin{equation}\label{eq:U^+s_grows_like_the_boundary_of_U}
	\#\left(\Lambda \cap U^{+s}\right) \le \#(\Lambda \cap U) + k^d\cdot c(s)\cdot\mu_{d-1}(\partial U).\notag
	\end{equation}  
\end{lem}

\begin{proof} 
	Fix $s$. For a set $U\in\QQ_d^*$ consider the set $V_s$ of all cubes $C(x)\in\QQ_d$ that intersect the set $U^{+s}\smallsetminus U$, and let $V_s^{(k)}$ be the set of cubes from $\frac1k \QQ_d$ inside $V_s$. then 
	$\#V_s^{(k)} = k^d \cdot \#V_s$.
	Clearly $\bigcup V_s \subset (\partial U)^{+(s+\sqrt{d})}$, and hence 
	\begin{equation}\label{eq:Lemma3.3_1}
	\#(\Lambda \cap (U^{+s}\smallsetminus U))\le \#V_s^{(k)} = 
	k^d \cdot \#V_s = 
	k^d \cdot\mu_d\left(\bigcup V_s\right) \le 
	k^d \cdot\mu_d\left( (\partial U)^{+(s+\sqrt{d})} \right).
	\end{equation}
	For $U\in\QQ_d^*$ and $b>0$, combining  Lemmas 2.1 and 2.2 of \cite{Laczk} one obtains 
	\begin{equation}\label{eq:Lemma3.3_2}
	\mu_d\left( (\partial U)^{+b} \right) \le 2\cdot 3^{d-1}(1+2b\sqrt{d})^d\cdot \mu_{d-1}(\partial U).
	\end{equation}
	Since $	\#(\Lambda \cap U^{+s}) = \#(\Lambda \cap U) +	\#(\Lambda \cap (U^{+s}\smallsetminus U))$, combining \eqref{eq:Lemma3.3_1} and \eqref{eq:Lemma3.3_2} with $b=s+\sqrt{d}$ implies the assertion.
\end{proof}

\begin{proof}[Proof of Theorem \ref{thm:non_BD_criterion}]
	Let $A_m\in\QQ_d^*$ be a sequence satisfying \eqref{eq:non_BD_condition}, and assume that there exists a BD-map $\phi:\Lambda_1 \to \Lambda_2$. Proposition \ref{prop:BD_criterion} then implies that there is a constant $s>0$ so that \eqref{eq:Hall's_condition_with cubes} holds for every $U\in \QQ_d^*$.
	Let $m$ be large enough so that 
	\[\absolute{ \#(\Lambda_1 \cap A_m) - \#(\Lambda_2 \cap A_m) } > k^d\cdot c(s)\cdot \mu_{d-1}(
	\partial A_m), \]
	where $k$ is small enough so that every box in $\frac1k \QQ_d$ contains at most one element of $\Lambda_1$ and at most one element of $\Lambda_2$, and where $c(s)$ is as in Lemma \ref{lem:boundary_est.}. Assume, without loss of generality, that
	\[\#\left(\Lambda_1 \cap A_m\right) > \#\left(\Lambda_2 \cap A_m\right) + k^d\cdot c(s)\cdot\mu_{d-1}(
	\partial A_m), \]
	passing to a subsequence if needed. By Lemma \ref{lem:boundary_est.} we have 
	\[\#\left(\Lambda_1 \cap A_m\right) > \#\left(\Lambda_2 \cap A_m^{+s}\right),\] 
	contradicting \eqref{eq:Hall's_condition_with cubes}, and so there is no BD-map between $\Lambda_1$ and $\Lambda_2$. 
\end{proof}


\section{BD non-equivalent sets associated with a common matrix }\label{sec:periodic+nonperiodic_from_one_matrix}
This chapter contains the proof of Theorem \ref{thm:period+nonperiodic_same_matrix}. Denote by $S$ a $1\times 1$ square and by $R$ a $3\times 1$ rectangle in $\R^2$, and let $\FF=\{S,R\}$. Consider the substitution rules $\varrho_1$ and $\varrho_2$ on $\FF$, as defined in Figures \ref{fig:substitutionrule1} and \ref{fig:substitutionrule2}. 

\begin{figure}[ht!]
	\includegraphics[scale=0.8]{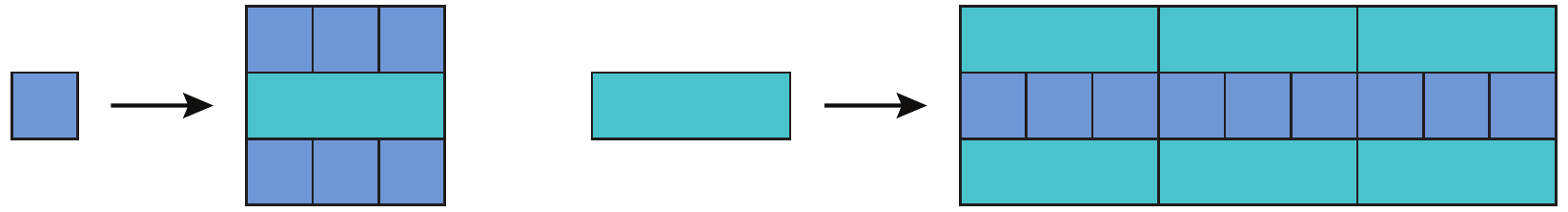}\caption{\label{fig:substitutionrule1}The substitution rule $\varrho_1$.}
\end{figure}

\begin{figure}[ht!]
	\includegraphics[scale=0.8]{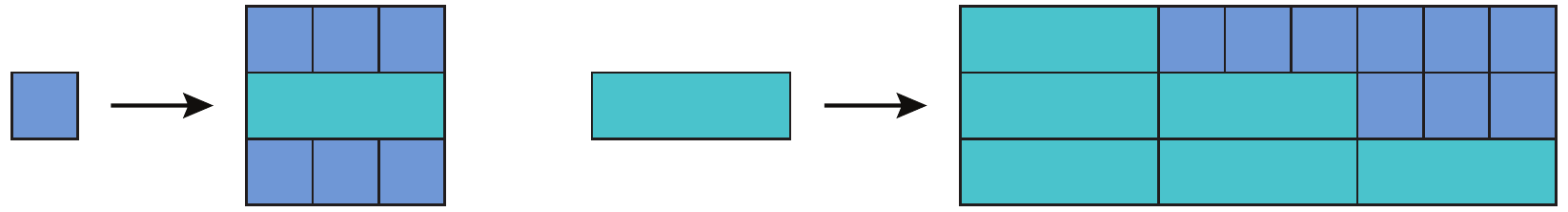}\caption{\label{fig:substitutionrule2}The substitution rule $\varrho_2$.}
\end{figure}

Observe that the inflation factor for both $\varrho_1$ and $\varrho_2$ is $\xi=3$. Set
\begin{align}\label{eq:matrix_info}
&M\df M_{\varrho_1}=M_{\varrho_2}=\begin{pmatrix}
6 & 9 \\ 1 & 6
\end{pmatrix},\quad \lambda_1=\xi^d=9, \lambda_2=3, \\
&
u_1=\begin{pmatrix}
\mu_2(S) & \mu_2(R)
\end{pmatrix}=\begin{pmatrix}
1 & 3
\end{pmatrix} \text{ and } v_1=\begin{pmatrix}
3 \\ 1
\end{pmatrix},
v_2=\begin{pmatrix}
-3 \\ 1
\end{pmatrix}.\notag
\end{align} 
Note that $u_1M=\lambda_1u_1$, and that $Mv_1=\lambda_1v_1$ and $Mv_2=\lambda_2v_2$. 

Denote by $S_m$ a $1\cdot 3^{m-1} \times 1\cdot3^{m-1}$ square and by $R_m$ a $3\cdot 3^{m-1} \times 1\cdot3^{m-1}$ rectangle in $\R^2$. These are called squares and rectangles of \emph{generation $m$}, and note that  $R$ and $S$ are of generation $1$. Apply  $\varrho_2$ to a rectangle $R$ exactly $m$ times to define a patch $\varrho_2^m(R)$ supported on a generation $m+1$ rectangle. Define the patch $P_m$ as the set of tiles in $\varrho_2^m(R)$ supported under the main NW-SE diagonal, as illustrated in Figure \ref{fig:patches p_m}.
	
	\begin{figure}[ht!]
		\includegraphics[scale=0.8]{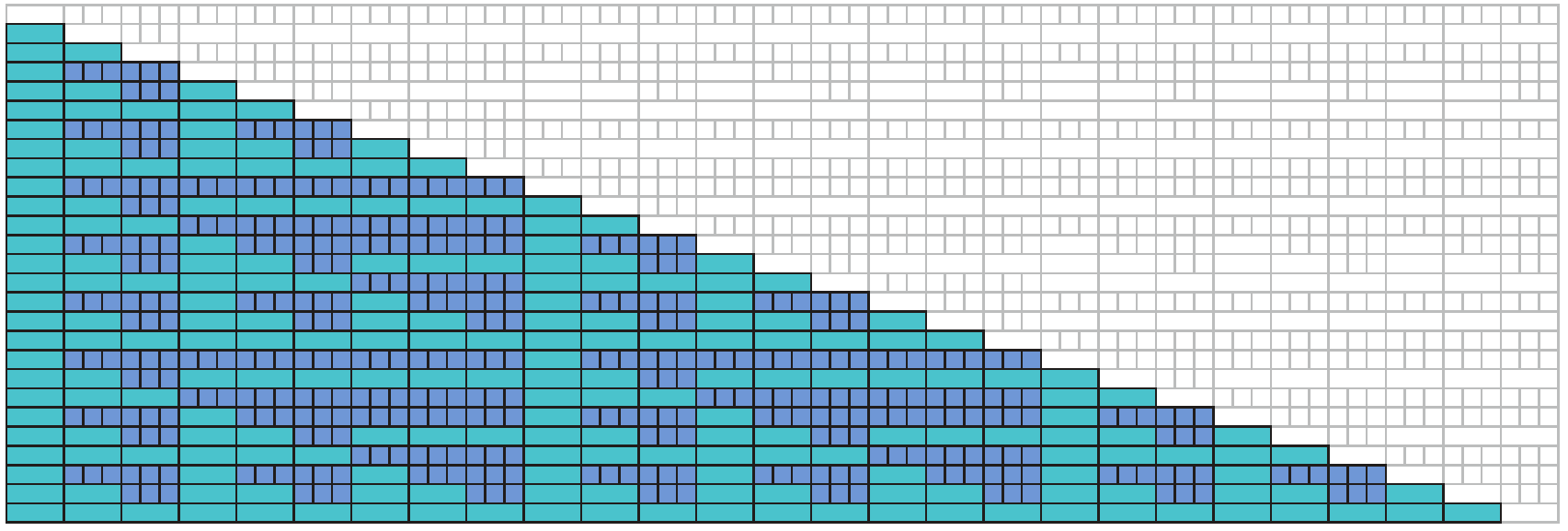}\caption{\label{fig:patches p_m}The patch $P_3$ as a sub-patch of $\varrho_2^3(R)$.}
	\end{figure}
	
We show that $\varrho_1, \varrho_2$ satisfy the assertions of Theorem \ref{thm:period+nonperiodic_same_matrix}. The heart of the proof is given in the following lemma. 

\begin{lem}\label{lem:P_M_counting}The number of tiles in $P_m$, the measure of its support and the measure of its boundary are given by
	\begin{align}\label{eq:number_of_tiles_in_P_m+volP_m}
	&\#P_m =  9^m - 3^m(m+1)\\
	&\mu_2(\supp(P_m)) =  \frac{3}{2}\left(9^m - 3^m\right), \quad \mu_1(\partial\supp(P_m)) =  8\cdot3^m-8\notag.
	\end{align}
\end{lem}

\begin{proof}
	The support of $P_m$ can be described as a disjoint union of rectangles, consisting of $3^j$ rectangles of generation $m-j+1$ for every $j = 1,\ldots,m$, as illustrated in Figure \ref{Fig:suppp_3}.
	
	\begin{figure}[ht!]
		\includegraphics[scale=0.8]{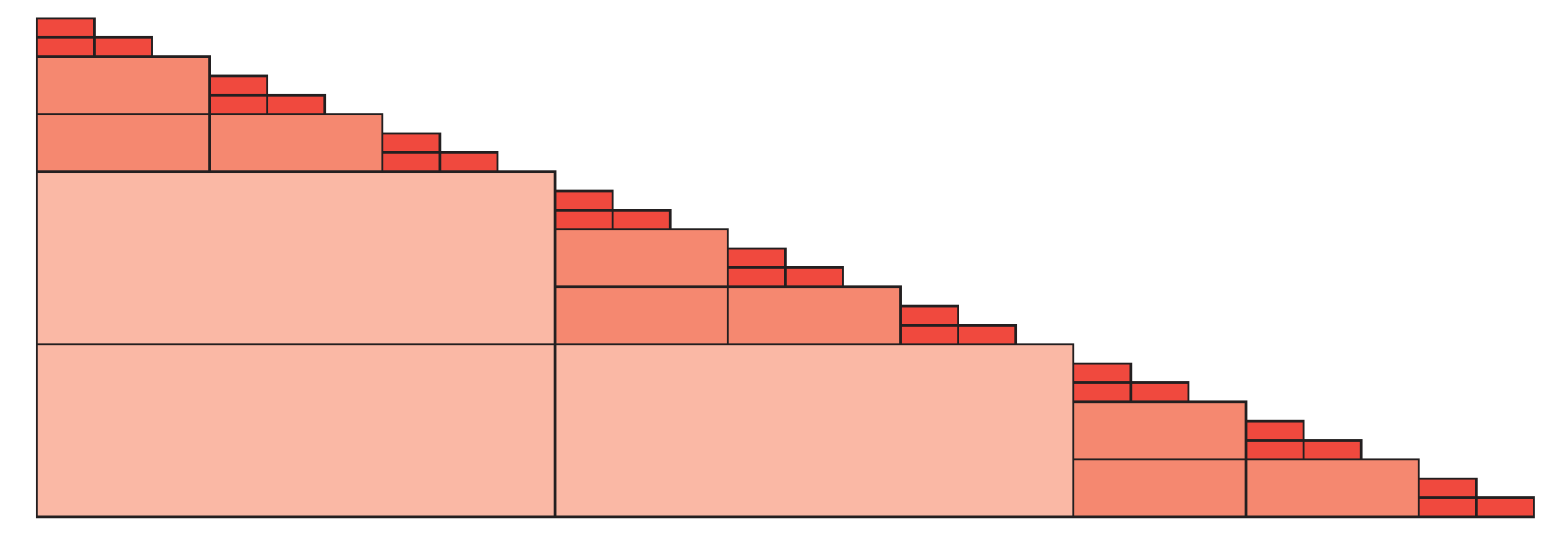}
		\caption{\label{Fig:suppp_3}Decomposition of $\supp(P_3)$ into rectangles of generations $1, 2$ and $3$.}
	\end{figure}

	For $j = 1,\ldots,m$, any such generation $j$ rectangle is the support of a translated patch of the form $\varrho_2^{j-1}(R)$. The number of squares in $\varrho_2^{j-1}(R)$ is given by the first coordinate of $M^{j-1}\binom{0}{1}$, and the number of rectangles by the second coordinate. Therefore $\#P_m$ is given by the sum of the coordinates of
	\begin{equation}\label{eq:number_of_tiles_in_P_m-vector} 
	\sum_{j=0}^{m-1} 3^{m-j}M^j\binom{0}{1} =\frac{3^m}{2}\left(\frac{3^m-1}{2}\binom{3}{1}+m\binom{-3}{1}\right), 
	\end{equation}
	where the equality follows from the fact that $\binom{0}{1}=\frac{1}{2}(v_1+v_2)$, with $v_1,v_2$ the eigenvectors of $M$ described in \eqref{eq:matrix_info}. Let $n_m(S)$ and $n_m(R)$ be the number of squares and rectangles in $P_m$, respectively. By \eqref{eq:number_of_tiles_in_P_m-vector}
	\begin{equation}
	n_m(S)=\frac{3}{4}\left(9^m-3^m(2m+1)\right),  n_m(R)=\frac{1}{4}\left(9^m+3^m(2m-1)\right).\notag
	\end{equation}
	Since $\#P_m=n_m(S)+n_m(R)$ and $\mu_2(\supp(P_m))=n_m(S)+3n_m(R)$ the first two assertions follow, and the third follows directly from the definition of $P_m$.
\end{proof}
	
\begin{lem}\label{lem:P_M_existence}
	Any tiling in $\XX_{\varrho_2}$ contains a translated copy of $P_m$ for any  $m\in\N$.
\end{lem} 

\begin{proof}
	By definition of $\XX_{\varrho_2}$, every tiling in $\XX_{\varrho_2}$ contains a translated copy of $\varrho_2^m(R)$ for any $m\in\N$, which contains $P_m$ as a sub-patch.
\end{proof}

\begin{proof}[Proof of Theorem \ref{thm:period+nonperiodic_same_matrix}]
First observe that assertion (i) holds, since any tiling defined by the substitution rule $\varrho_1$ is periodic, with two linearly independent periods $(3,0), (0,2)\in\R^2$, and a fundamental domain as illustrated in Figure \ref{fig: period}.

	\begin{figure}[ht!]
	\includegraphics[scale=0.8]{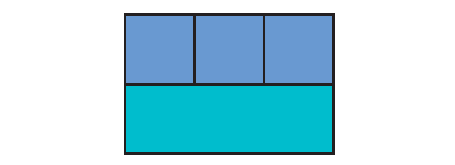}\caption{\label{fig: period}The fundamental domain of a periodic tiling associated with the substitution rule $\varrho_1$.}
\end{figure}

It is left to prove assertion (ii). Let $\TT\in\XX_{\varrho_2}$ and let $\Lambda$ be an associated Delone set. By \eqref{eq:alpha_def}, the natural density of $\Lambda$ is  $\alpha=\frac{\langle\mathbf{1},v_1\rangle}{\langle u_1,v_1 \rangle}=\frac{2}{3}$, and so by Corollary \ref{cor:different_densities_implies_not_BD} there is no BD-map $\phi:\Lambda \to \beta\Z^d$ for any $\beta \neq \sqrt{\alpha}$.

By Lemma \ref{lem:P_M_existence}, $\TT$ contains translated copies of all patches $P_m$, and we fix $A_m=\supp(P_m)$ for some translated copy of $P_m$ in $\TT$. Note that the inequality $\absolute{\#(\sqrt{\alpha}\Z^2 \cap A_m) - \alpha\mu_2(A_m)} \le C\mu_1(\partial A_m)$ holds for any $m\in\N$, with some fixed constant $C$. In order to show that there is no BD-map $\phi:\Lambda\to\sqrt{\alpha}\Z^2$, in view of Theorem \ref{thm:non_BD_criterion}, it suffices to show that 
\begin{equation}\label{eq:non_BD_to_Z^d_condition}
\lim_{m\to\infty} \frac{| \#(\Lambda \cap A_m) - \alpha\mu_2(A_m) |}{\mu_1(\partial A_m)} = \infty.
\end{equation}
Indeed, by \eqref{eq:number_of_tiles_in_P_m+volP_m} we have 
\[\absolute{\#P_m-\alpha \mu_2(\supp(P_m))} = m3^m, \quad \mu_{1}(\partial\supp(P_m)) \le 8\cdot 3^m\]
and \eqref{eq:non_BD_to_Z^d_condition} follows. 
\end{proof}


\section{Generating continuously many BD non-equivalent sets}\label{sec:continuously_many_non-BD_sets}
This chapter contains the proof of Theorem \ref{thm:continuously_many_things}. Let $\sigma_1$ and $\sigma_2$ be the substitution rules on $\FF$ defined in Figures \ref{fig:SubstitutionRule2diagonals} and \ref{fig:SubstitutionRule2diagonals2}, where once again $\FF=\{S,R\}$ with $S$ a $1\times 1$ square and $R$ a $3\times 1$ rectangle. 

\begin{figure}[ht!]
	\includegraphics[scale=0.8]{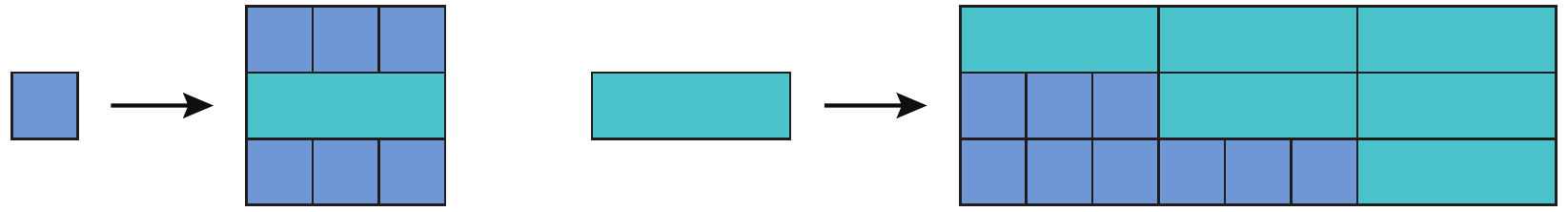}\caption{\label{fig:SubstitutionRule2diagonals}The substitution rule $\sigma_1$.}
\end{figure} 

\begin{figure}[ht!]
	\includegraphics[scale=0.8]{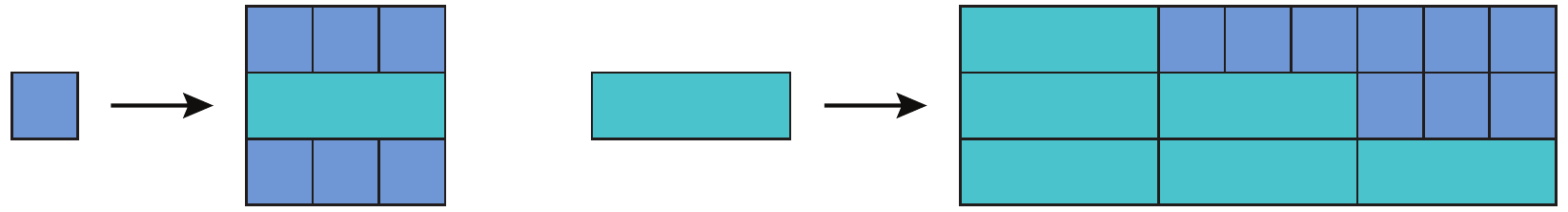}\caption{\label{fig:SubstitutionRule2diagonals2}The substitution rule $\sigma_2$.}
\end{figure} 

 Note that $M_{\sigma_1}=M_{\sigma_2}$ is the substitution matrix studied in \S \ref{sec:periodic+nonperiodic_from_one_matrix}, and so we let $M, v_1, v_2, \lambda_1, \lambda_2$ and $\xi$ be as in \eqref{eq:matrix_info}. Denote $\AA=\{1,2\}$ and $\Omega=\AA^\N$. 

Similarly to our definition of $P_m$ in \S \ref{sec:periodic+nonperiodic_from_one_matrix}, given a finite word $w\in\AA^m$ we define the patch $P_m^w$ to be the sub-patch of $w.R$ that consists of all tiles in $w.R$, which are situated under the NW-SE main diagonal, where the action of a finite word on a patch is as defined in \eqref{eq:action_of_a_finite_word}. For example, the patch corresponding to the word $w=112\in \AA^3$ is illustrated in Figure \ref{fig:Patch211}.

\begin{figure}[ht!]
	\includegraphics[scale=0.8]{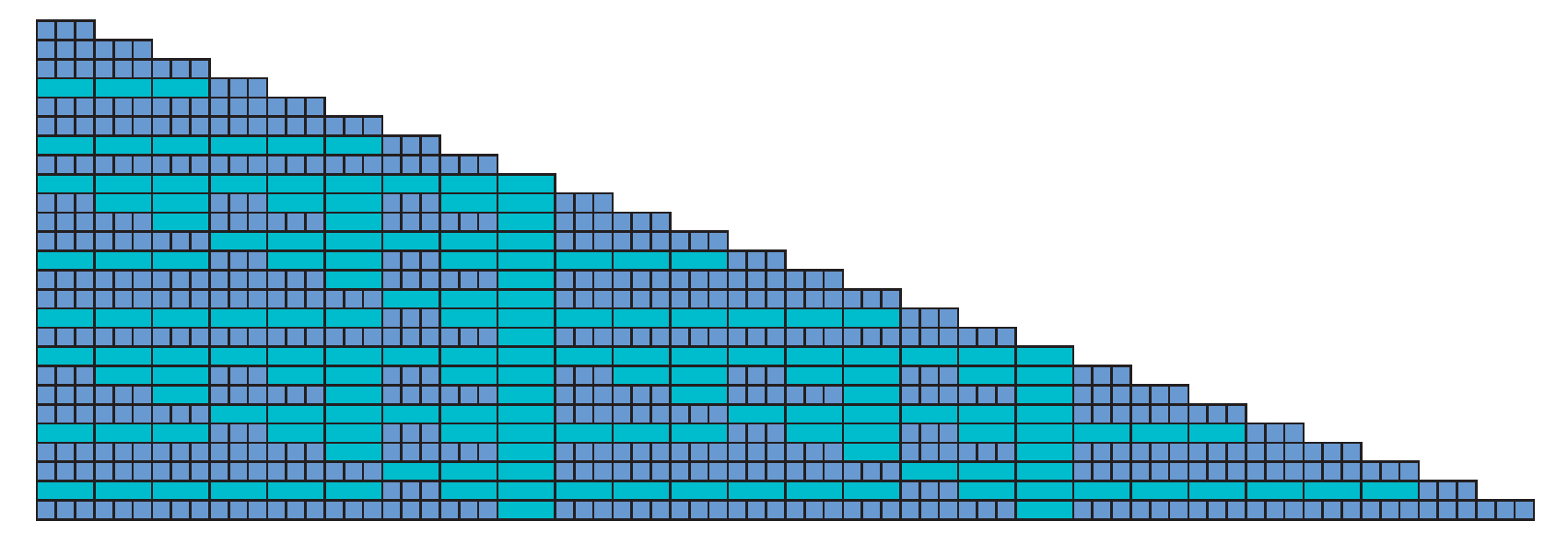}\caption{\label{fig:Patch211}The patch $P_3^{112}$.}
\end{figure}

In fact, since $\sigma_2=\varrho_2$ the patch $P_m$ described in  
\S \ref{sec:periodic+nonperiodic_from_one_matrix} is the same as the patch associated with the word $w=2^m=2\cdots2$, that is,
$P_m=P_m^{2^m}$.

As was done for the support of $P_m$ and illustrated in Figure \ref{Fig:suppp_3}, we decompose the support of $P_m^w$ into disjoint squares and rectangles of generations $j = 1,\ldots,m$. This decomposition depends on the word $w$, where for every $1\le j\le m$ the number of squares and rectangles of generation $j$ in the decomposition depends on $w_j$, the $j$'th letter of $w$, in the following way:

\begin{equation}\label{eq2:P_m-description}
\supp(P_m^w) \text{ contains: } 
\begin{cases}
9\cdot 3^{m-j}\text{ squares }S_j, &\text{ if } w_{j}=1, \text{ or}\\
3\cdot 3^{m-j}\text{ rectangles }R_j, &\text{ if } w_{j}=2
\end{cases}
\end{equation}
For example, the decomposition of the support of the patch $P_3^{112}$, associated with the word $w=112$, into squares and rectangles is illustrated in Figure \ref{fig:supp211}.

\begin{figure}[ht!]
 	\includegraphics[scale=0.8]{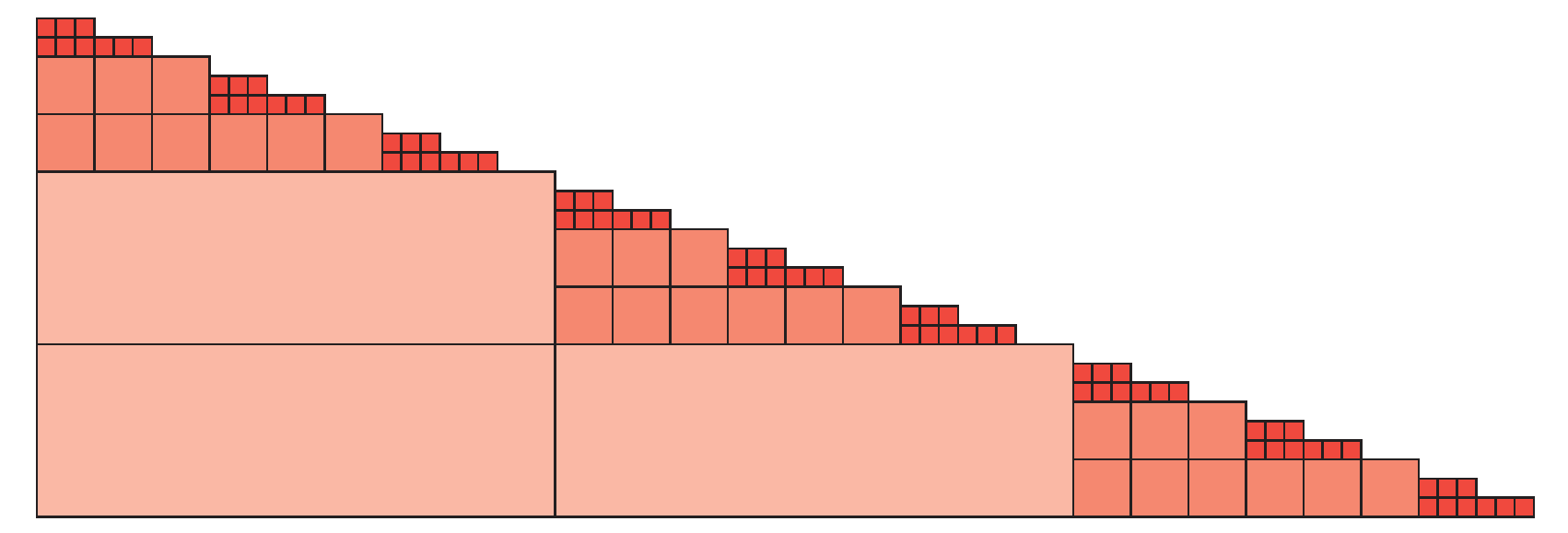}\caption{\label{fig:supp211}The decomposition of $\supp(P_3^{112})$ into squares of generations $1$ and $2$, and rectangles of generation $3$.}
\end{figure}

\begin{lem}\label{lem2:P_m^sigma_counting}
	Let $w\in\AA^m$. The number of tiles in $P_m^w$ is given by
	\begin{align}\label{eq2:number_of_tiles_in_P_m}
	\#P_m^w = 9^m - 3^m \left[1+\sum_{j=1}^{m}
	\begin{cases}
	1, &\text{if }\: w_j=1\\
	-1,&\text{if }\: w_j=2
	\end{cases}
	\right]
	\end{align}
\end{lem}
\begin{proof}
	As in the proof of Lemma \ref{lem:P_M_counting}, each of the squares of generation $j$ in the decomposition of $\supp(P_m^w)$ is the support of a patch $w.S$, where $w\in\{1,2\}^{j-1}$, and analogously for rectangles. The number of tiles in each square $S_j$ is given by the sum of the coordinates of the vector $M^{j-1}\binom{1}{0}$, and the number of tiles in rectangles $R_j$ is given by the sum of the coordinates of the vector $M^{j-1}\binom{0}{1}$. Therefore, in view of \eqref{eq2:P_m-description}, $\#P_m^w$ is given by the sum of the coordinates of
	\begin{align}\label{eq2:no._of_tiles_in_P_n_vector}
	\sum_{j=1}^{m} \begin{cases}
	9\cdot 3^{m-j}M^{j-1}\binom{1}{0},&\text{if }\: w_j=1  \\
	3\cdot 3^{m-j}M^{j-1}\binom{0}{1},&\text{if }\: w_j=2 
	\end{cases}. 
	\end{align}
	We plug in $\binom{1}{0}=\frac{1}{6}\left(v_1-v_2 \right)$ and $\binom{0}{1}=\frac{1}{2}\left(v_1+v_2 \right)$  in \eqref{eq2:no._of_tiles_in_P_n_vector} to obtain
	\begin{equation}
	=\frac{3^m}{2}\left[
	\frac{3^m-1}{2}\cdot\binom{3}{1} + \binom{-3}{1}\cdot\sum_{j=1}^{m}
	\begin{cases}
	-1 ,&\text{if }\:  w_j=1  \\
	1 ,&\text{if }\:  w_j=2 
	\end{cases}\right]. \label{eq2:number_of_tiles_in_P_m-vector}\notag
	\end{equation}
Summing the coordinates, the assertion follows. 
\end{proof}

Next, we show that the sum in \eqref{eq2:number_of_tiles_in_P_m} can be used to approximate any real $\gamma\in[-1,1]$. These approximations depend on $m$, and possess an inductive property, which is later used to concoct specific patches that appear in the proof of Theorem \ref{thm:continuously_many_things}.

\begin{lem}\label{lem2:finite_words_approx._gamma}
	For any real $\gamma\in[-1,1]$ there exists an infinite word $w^\gamma\in\Omega$, so that for any $m\in\N$ 
	\begin{equation}\label{eq2:sum_close_to_any_gamma}
		\absolute{\gamma - \frac{1}{m}\sum_{j=1}^{m}
		\begin{cases}
		1, &\text{if }\: w^\gamma_j=1\\
		-1,&\text{if }\: w^\gamma_j=2
		\end{cases}}\le \frac{1}{m}.
	\end{equation}
\end{lem}	
	
\begin{proof}
	We define $w^\gamma$ inductively. For $m=1$ set $w^\gamma_1=1$ if $\gamma\ge 0$ and $w^\gamma_1=2$ otherwise. Assume that the first $m$ letters in $w^\gamma$ are defined and that \eqref{eq2:sum_close_to_any_gamma} holds for $m$. Then there exists a constant $-\frac{1}{m}\le c \le \frac{1}{m}$ for which
	\begin{equation*}
	\frac{1}{m}\sum_{j=1}^{m}
	\begin{cases}
	1, &\text{if }\: w^\gamma_j=1\\
	-1,&\text{if }\: w^\gamma_j=2
	\end{cases}=\gamma+c.	
	\end{equation*}
	Since $\absolute{\gamma - mc} \le 2$, there is a choice of $u\in\{\pm1\}$ for which $\absolute{\gamma - mc-u}\le1$, and we set $w^\gamma_{m+1}=1$ if $u=1$ and 
	$w^\gamma_{mj+1}=2$ otherwise. It follows that
	\begin{equation}
	\absolute{\gamma - \frac{1}{m+1}\sum_{j=1}^{m+1}
		\begin{cases}
		1, &\text{if }\: w^\gamma_j=1\\
		-1,&\text{if }\: w^\gamma_j=2
		\end{cases}} \\
	=\absolute{\gamma - \frac{1}{m+1} \left(m\cdot (\gamma+c) +u\right)}\\
	\le \frac{1}{m+1},\notag
	\end{equation}
	and so $w^\gamma$ satisfies \eqref{eq2:sum_close_to_any_gamma} for $m+1$.
\end{proof}
	
We now turn to the construction of a family of continuously many tilings, so that Delone sets associated with distinct tilings are BD non-equivalent.

Fix a rectangular tile $R$ centered at the origin. Observe that applying the substitution rules $\sigma_1,\sigma_2$ to $R$, the tile containing the origin in both $\sigma_1(R)$ and in $\sigma_2(R)$ is again $R$. Therefore if $w'$ is a prefix of $w$ then the patches $w.R$ and $w'.R$ coincide on the support of $w'.R$, and so for every $\gamma\in[-1,1]$ and $m\in\N$, the patch $w^\gamma(m+1).R$ contains the patch $w^\gamma(m).R$ as a sub-patch. Thus 
\begin{equation}
\TT[\gamma] \df \bigcup_{m\in\N} w^\gamma(m).R \notag
\end{equation}
is a tiling of the plane, and clearly $\TT[\gamma]\in\XX_{\Sigma,w^\gamma}$, with $\Sigma=(\sigma_1,\sigma_2)$.

Next, and in analogy to \S4, for any $m\in\N$ we define $A_m$ to be the copy of the support of a patch of the form $P^w_m$ situated at the bottom left corner of a generation $m+1$ rectangle centered at the origin, where $w$ is some arbitrary word in $\AA^m$. As a result of our choice of $\sigma_1$ and $\sigma_2$, this defines a sequence of sets in $\QQ_2^*$. Moreover, for every $m\ge 2$ and every $\gamma\in[-1,1]$ the set $A_m$ is the support of the patch $P_{m}^{w^\gamma(m)}$ in $\TT[\gamma]$. \\

\noindent \textbf{Remark.} We provide the explicit construction above for clarity reasons only. In fact, for any $\gamma\in[-1,1]$ the existence of a tiling with the properties shown above to hold for $\TT[\gamma]$ can be established in a more general setting by the standard fact that the tiling space $\XX_{\Sigma,w^\gamma}$ is closed under translations and compact. 

\begin{proof}[Proof of Theorem \ref{thm:continuously_many_things}]
	Let $\gamma_1\neq\gamma_2$ be two real numbers in $[-1,1]$. We show that any two Delone sets $\Lambda_1, \Lambda_2$ associated with $\TT[\gamma_1], \TT[\gamma_2]$, respectively, are BD non-equivalent. Since any other choice would give BD equivalent Delone sets, we assume without loss of generality that $\Lambda_1, \Lambda_2$ are obtained from $\TT[\gamma_1], \TT[\gamma_2]$ by placing a point at the center of each tile. 
	We now show that the sets $A_m\in\QQ_2^*$ defined above satisfy condition \eqref{eq:non_BD_condition} of Theorem \ref{thm:non_BD_criterion}. 
	
	Indeed, by Lemmas \ref{lem2:P_m^sigma_counting} and \ref{lem2:finite_words_approx._gamma} we have
	\begin{align}
	&
	\absolute{\#(\Lambda_1\cap A_m) - \#(\Lambda_2\cap A_m)}= \absolute{\#P_m^{w^{\gamma_1}(m)} - \#P_m^{w^{\gamma_2}(m)}}  \notag \\
	&=3^m\absolute{
		\left[ \sum_{j=1}^m
		\begin{cases}
		1, &\text{if }\: w^{\gamma_1}_j=1\\
		-1,&\text{if }\: w^{\gamma_1}_j=2
		\end{cases} \right]   -
		\left[ \sum_{j=1}^m
		\begin{cases}
		1, &\text{if }\: w^{\gamma_2}_j=1\\\label{eq:numerator_estimate_for_gammas2}
		-1,&\text{if }\: w^{\gamma_2}_j=2
		\end{cases} \right]	} \\
	&\ge 3^m  \absolute{m(\gamma_1-\gamma_2) - c_{\gamma_1,\gamma_2}},\notag
	\end{align}
	where $c_{\gamma_1,\gamma_2}$ is some constant with $\absolute{c_{\gamma_1,\gamma_2}}\le2$. Note that $A_m$ is the support of a translate of a patch $P_m$ as in \S4, and so it follows from  \eqref{eq:number_of_tiles_in_P_m+volP_m} that $\mu_1(\partial A_m)$ is bounded from above by $8\cdot 3^m$. Combining this with \eqref{eq:numerator_estimate_for_gammas2}, and using the assumption that $\gamma_1 \neq \gamma_2$, we deduce
	\[\lim_{m\to\infty}\frac{\absolute{\#(\Lambda_1\cap A_m) - \#(\Lambda_2\cap A_m)}}{\mu_1(\partial A_m)} = \infty.\]	
	By Theorem \ref{thm:non_BD_criterion} and Corollary \ref{cor:different_densities_implies_not_BD} there is no BD-map $\phi:\Lambda_1 \to \alpha\Lambda_2$ for any scaling constant $\alpha>0$, and so $\Lambda_1$ and $\Lambda_2$ are BD non-equivalent.
	
	To complete the proof, recall that the substitution rules $\sigma_1,\sigma_2\in\Sigma$ have the same substitution matrix. Therefore, standard counting arguments (see e.g. \cite{Solomon11}) imply that every Delone set associated with a tiling in $\XX_{\Sigma,\Omega}$ is BL equivalent to a lattice. 
\end{proof}

\begin{cor}\label{cor:rotation_by_pi_not_BD}
	There exists a Delone set $\Lambda$ corresponding to $\sigma_1$ with $\Lambda \not\bd -\Lambda$.
\end{cor}
\begin{proof}
Note that $\TT\mapsto-\TT$ is a bijection between $\XX_{\sigma_1}$ and $\XX_{\sigma_2}$. Identify the substitutions $\sigma_1$ and $\sigma_2$ with the ``mixed substitutions'' defined by the constant words $w_1=1^\infty$ and $w_2=2^\infty$, which in turn are associated with the real numbers $\gamma_1=1$ and $\gamma_2=-1$.  The proof follows from the fact that $1\neq-1$.
\end{proof}

The proof of Theorem \ref{thm:arbitrary-many-things} follows from the next result.
\begin{lem} \label{lem2:n+1-deterministic-subst}
	Let $q\in\N$, and for $p=0,\ldots, q$ denote $\varrho_p \df \sigma_1^p \sigma_2^{q-p}$, where $\sigma_1\sigma_2=\sigma_1\circ\sigma_2$. There exist tilings $\TT_0,\ldots,\TT_q$ defined by $\varrho_0,\ldots,\varrho_q$ with pairwise non-equivalent associated Delone sets.
\end{lem}
\begin{proof}[Sketch of proof] 
	The arguments are similar to those discussed above. Fix $p\in\{0,\ldots,q\}$. The substitution rule $\varrho_p$ corresponds to the infinite word 
	\[ w^p \df 1^p 2^{q-p} \,  1^p 2^{q-p} \,  1^p 2^{q-p} \,  1^p 2^{q-p} \, \cdots, \]
	where $a^k$ denotes $k$ consecutive appearances of $a\in\{1,2\}$. Recall that for any $m\in\N$ there exist $\ell \ge 0$ and  $0\le r\le q-1$ so that  $m=\ell q + r$. Then
	
	\begin{equation*}
\frac{1}{m}\sum_{j=1}^{\ell q}
\begin{cases}
1, &\text{if }\:  w^p_j=1\\
-1,&\text{if }\:  w^p_j=2
\end{cases} =\frac{\ell q}{m} \cdot \frac{1}{\ell q}(\ell p - \ell (q-p))= \left(1-\frac{r}{m}\right)\gamma_p
	\end{equation*}
where $\gamma_p = \frac{2p-q}{q}$. It follows that 
	\begin{equation*}
	\absolute{\gamma_p - \frac{1}{m}\sum_{j=1}^{m}
		\begin{cases}
		1, &\text{if }\: w^p_j=1\\
		-1,&\text{if }\: w^p_j=2
		\end{cases}} 
	= \absolute{ \frac{r}{m} \gamma_p-  \frac{1}{m}\sum_{j=1}^{r}
		\begin{cases}
		1, &\text{if }\: w^p_j=1\\
		-1,&\text{if }\: w^p_j=2
		\end{cases}}\le \frac{2q}{m}.  
	\end{equation*}
	
	Adjusting the proof of Lemma \ref{lem2:finite_words_approx._gamma} so that the right hand side of equation \eqref{eq2:sum_close_to_any_gamma} has numerator $2q$ instead of $1$, the above shows that the word $w^p$ induces adjusted approximations for $\gamma_p$. Using these approximations and similarly to the computation appearing in \eqref{eq:numerator_estimate_for_gammas2}, for distinct $p_1, p_2\in\{0,\ldots,q\}$ we obtain 
	\begin{equation*}
	\absolute{\#P_m^{w^{\gamma_{p_2}}(m)} - \#P_m^{w^{\gamma_{p_1}}(m)}} 
	\ge 3^m  \absolute{\frac{m}{2q}(\gamma_{p_2}-\gamma_{p_1}) - c_{\gamma_{p_2},\gamma_{p_1}}}.\notag
	\end{equation*}
	Since $q$ is fixed and $p_1\neq p_2$, we have
	\[\lim_{m\to\infty}\frac{\absolute{\#P_m^{w^{\gamma_{p_2}}(m)} - \#P_m^{w^{\gamma_{p_1}}(m)}} }{\mu_1(\partial A_m)} = \infty.\]	
	By arguments described in the proof of Theorem \ref{thm:continuously_many_things}, this settles the lemma.
\end{proof}

\section{Acting from the left vs. acting from the right}\label{sec:right-left}
Consider a mixed substitution system $\Sigma=(\sigma_1,\ldots,\sigma_k)$ on a set of prototiles $\FF$ in $\R^d$, and recall the definitions for mixed substitution tilings and spaces given in \S \ref{sec:preliminaries}. The action of a finite word $a=(a_1,\ldots,a_m)\in\AA^m$  on a patch $P$ can be viewed as an action from the left. Naturally, one may also consider an analogous \textit{action from the right}, defined by
\begin{equation*}\label{eq:right_action_of_a_finite_word}
P.\act{a} \df ( ( \cdots (P)\sigma_{a_1} \cdots)\sigma_{a_{m-1}})\sigma_{a_m} .
\end{equation*}
Here $(P)\sigma_i$ stands for the patch obtained by applying $\sigma_i$ to all tiles in $P$, that is $(P)\sigma_i=\sigma_i(P)$, and the notation $(P)\sigma_i$ is used merely to emphasize that the action considered is from the right.

Set $\AA=\{1,\ldots, k\}$ and $\Omega=\AA^\N$. Given $w\in\Omega$, denote 
\[\RR_w \df \left\{ T.\act{w(m)} \mid T\in \FF, m\in\N  \right\}.\]
The \emph{(right)-mixed substitution tiling space} $\YY_{\Sigma,w}$ is the collection of tilings $\TT$ of $\R^d$ with the property that for each patch $P \subset \TT$ there is a patch $P' \in \RR_w$ that contains a translated copy of $P$ as a sub-patch. Denote  $\YY_{\Sigma,\Omega}=\bigcup_{w\in\Omega} \YY_{\Sigma,w}$. 

Each of the two definitions of mixed substitution tilings has its advantages. For example, when acting from the left, the structure of the finite generation patches is clear. When acting from the right, as done recently in \cite{SchmiedingTrevino}, it is clear how an infinite word $w\in \Omega$ acts. One reason why the left action is popular is the fact that, unlike tilings spaces  $\YY_{\Sigma,w}$, the tiling spaces $\XX_{\Sigma,w}$ are minimal with respect to the translation action by $\R^d$, see \cite[Proposition 6]{Durand}, and
\cite[Proposition 2.12]{Rust} for the one-dimensional case. 

Recall the definition of uniform primitivity given in \S \ref{sec:preliminaries}. A nice property of $\YY_{\Sigma,\Omega}$ is that for uniformly primitive systems $\YY_{\Sigma,\Omega} = \YY_{\Sigma,w}$ for a generic word $w$, as we show in Lemma \ref{lem:YY_Omega=YY_w}. First, let $\nu$ be a probability measure on $\AA$ with full support. Denote by $\widetilde{\nu}$ the product measure on $\Omega$ that arises from $\nu$. Given a finite word $a\in\AA^m$ for some $m\in\N$, clearly $\widetilde{\nu}$-almost every sequence in $\Omega$ must contain a (translated) copy of $a$. Since $\AA^*$ is countable, it follows that $\widetilde{\nu}$-almost every sequence in $\Omega$ contains all finite words with letters from $\AA$. Denote by $\widetilde{\Omega}$ the set of such sequences.

\begin{lem}\label{lem:YY_Omega=YY_w}
	Let $\Sigma$ be uniformly primitive mixed substitution system. Then $\YY_{\Sigma,\Omega} = \YY_{\Sigma,w}$ for every $w\in\widetilde{\Omega}$. 
\end{lem}

\begin{proof}
	The inclusion $\YY_{\Sigma,\Omega} \supset \YY_{\Sigma,w}$ is obvious. For the second inclusion, let $\TT\in \YY_{\Sigma,\Omega}$ 
	and let $P\subset \TT$ be a patch. 
	Since $\TT\in \YY_{\Sigma,\Omega}$ there exists a word $w'\in\Omega$ for which $\TT\in\YY_{\Sigma,w'}$. Therefore, there is some patch $P'$ of the form $P' = T.\act{w'(n)}\in \RR_{w'}$ that contains $P$ as a sub-patch. By uniform primitivity of $\Sigma$, there exists $m_0\in\N$ so that for every word $a\in\AA^{m}$ with $m\ge m_0$ and every $S\in\FF$, the patch $S.\act{a}$ contains tiles of all types and in particular a copy of $T$. Since $w\in\widetilde{\Omega}$, there exists some $l\ge n+m_0$ such that $w'(n)$ is a suffix of $w(l)$. By definition of the action from the right, this implies that $P'' \df T.\act{w(l)}\in\RR_w$ contains $P'$ as a sub-patch, hence it also contains $P$ as a sub-patch.
	Since $P$ was arbitrary, by definition $\TT\in \YY_{\Sigma,w}$.
\end{proof}

\noindent \textbf{Remark.} Lemma \ref{lem:YY_Omega=YY_w} fails for the action from the left. In fact, assuming a simple condition of non-degeneracy on $\Sigma$, it is not hard to show that for any $w \neq u$ in $\Omega$ one has $\XX_{\Sigma,w} \nsubseteq \XX_{\Sigma,u}$. \\ 

The following lemma shows that Theorem  \ref{thm:continuously_many_things} also holds for $\YY_{\Sigma,\Omega}$. 

\begin{lem}\label{lem:tiling_spaces_are_equal}
	Let $\Sigma$ be a mixed substitution system. Then 
	$\XX_{\Sigma,\Omega} \subset \YY_{\Sigma,\Omega}$. 	
\end{lem}

\begin{proof}
	For a finite word $a = (a_1,\ldots,a_m)\in\AA^m$ let $\overline{a} =(a_m,\ldots,a_1)$.  Clearly, for every $m\in\N$, $T\in\FF$, and a word $w(m) \in\AA^m$, we have $\act{w(n)}.T = T.\act{\overline{w(n)}}$. Note that when acting from the right with $w(m) \in \AA^m$, the structure of the patches of the first $r$ generations in the resulting pattern is determined by the suffix of length $r$ of $w(m)$. Hence for every  ${\widetilde{w}}\in\widetilde{\Omega}$, and every $w\in\Omega$, one has  $\PP_{w}\subset\RR_{\widetilde{w}}$. By the definition of the tiling spaces, this implies that $\XX_{\Sigma,w} \subset \YY_{\Sigma,{\widetilde{w}}}$ for any $w\in\Omega$. Thus $\XX_{\Sigma,\Omega}\subset \YY_{\Sigma,{\widetilde{w}}}$, and hence $\XX_{\Sigma,\Omega} \subset \YY_{\Sigma,\Omega}$.
\end{proof}  

Since the system $\Sigma=(\sigma_1,\sigma_2)$ defined in \S \ref{sec:continuously_many_non-BD_sets} is clearly uniformly primitive, we  can combine the above with Theorem \ref{thm:continuously_many_things} to obtain the following result.

\begin{cor}
	 Let $\FF$ and $\Sigma=(\sigma_1,\sigma_2)$ be as in \S \ref{sec:continuously_many_non-BD_sets}. Then for $\widetilde{\nu}$-almost every $w\in\Omega=\{1,2\}^\N$, the space $\YY_{\Sigma,w}$ contains continuously many tilings with pairwise BD non-equivalent associated Delone sets. 	
\end{cor}

\end{document}